\documentclass{article}
\usepackage{amssymb}
\usepackage{amsthm}
\newtheorem{theorem}{Theorem}
\newtheorem*{md}{Model Example}\newtheorem*{exa}{Example}
\newtheorem*{ml}{Main Lemma}
\newtheorem{lemma}{Lemma}
\newtheorem{corollary}{Corollary}
\newtheorem{definition}{Definition}

\newcommand{\La}{ \Lambda}

\newcommand{\R}{{\mathbb R}}

\newcommand{\Z}{{\mathbb Z}}
\newcommand{\N}{{\mathbb N}}

\newcommand{\C}{{\mathbb C}}

\newcommand{\G}{\Phi}

\newcommand{\g}{\varphi}

\begin{document}
\begin{large}

\title{Discrete Translates in Function Spaces} 

\author{Alexander Olevskii\thanks{This author's research is supported in part  by the Israel Science  Foundation}\,  and  Alexander Ulanovskii}

\date{}\maketitle

\noindent A.O.: School of Mathematics, Tel Aviv University\\ Ramat Aviv,  69978 Israel. E-mail:
 olevskii@post.tau.ac.il

\medskip

\noindent A.U.: Stavanger University,  4036 Stavanger, Norway\\ E-mail: Alexander.Ulanovskii@uis.no

\begin{abstract}We construct 
a Schwartz function  $\varphi$ such that for every exponentially small perturbation of integers $\La$ 
the set of translates $\{\varphi(t-\lambda),\lambda\in\La\}$                                    spans the space $L^p(\R)$ for every $p>1$.
                                  This result remains true  for more general function spaces $X$,  whose  norm is "weaker" than $L^1$  (on bounded functions).
\end{abstract}

   \section{Introduction}

   \noindent{\bf 1. Completeness of translates}. In what follows we will use  the standard form of Fourier transform  \[\varphi(x)=\hat\Phi(x):=\int_\R e^{-2\pi i tx}\Phi(t)\,dt.\]

Classical Wiener's  theorems \cite{w}  provide  necessary and sufficient conditions on a function $\varphi$ whose  translates $\{\varphi(x-s), s\in\R\}$ span the space  $L^1(\R)$ or $L^2(\R)$:

\medskip

(i) The translates     of  $\g\in L^1(\R)$ span $L^1(\R)$  if and only if  $\hat\g$ does not vanish;

 (ii) The translates     of  $\g\in L^2(\R)$ span $L^2(\R)$       if and only if  $\hat\g$ is non-zero almost everywhere on $\R$.

\medskip

When $1<p<2$, Beurling \cite{B} proved that the translates of a function $\varphi\in L^p(\R)\cap L^1(\R)$ span $L^p(\R)$ provided
the  Hausdorff dimension of the zero set of  $\hat\varphi$ is  less than $ 2(p-1)/p$. However, this condition is not necessary. Moreover, the spanning property of the translates of  $\g\in L^p(\R)$ cannot be characterized in terms of the zero set of  $\hat\varphi$, see \cite{LO}.

\medskip\noindent{\bf 2. Discrete spectra}.
It is well known that sometimes even a {\it discrete} set of  translates may  span $L^p(\R).$

\begin{definition}  We say that a discrete set $\La\subset\R$ is a {\it $p-$spectrum} {\rm (}for translates{\rm )},   if there is a function $\g\in L^p(\R)$ such that the family of translates
 \[              \{\g(t-\lambda),   \lambda\in \La\}\]
        spans $L^p(\R)$.  Such a function $\g$ is called  {\it a $\La-$generator}.
\end{definition}

     For $p=2 $, using the Fourier transform one can re-formulate the definition above as follows:

      {\it A set $\La$ is a $2$-spectrum if and only if  there exists $\G\in L^2(\R)$ such that the  system
  \[\{\G(t) e^{2\pi i\lambda t}, \lambda\in \La\}\] spans the whole space $L^2(\R) $.}

  Using this, it is easy to see that the set of integers $\Z$ is not a $2$-spectrum.
        On the other hand, the following is true:

                 {\it  Every sequence $\La$ obtained by a "perturbation" of integers,
     \begin{equation}\label{ais}
\La=\{\lambda_n:=n+a_n, n\in\Z\},\quad a_n\ne0, n\in\Z,
\end{equation}where $a_n\to0$ as $|n|\to\infty,$ is a $2$-spectrum} (\cite{O}).

          Assume  the "perturbations" $a_n$ are exponentially small:
\begin{equation}\label{an}
|a_n|<Cr^{|n|}, \quad n\in\Z;\ \mbox{ for some }  C>0,0<r<1.
\end{equation}
      Then    a stronger result is true  (\cite{OU2}):

      {\it If $\La$ satisfies (\ref{ais}) and (\ref{an}), then a $\La$-generator $\g$ can be chosen to be a Schwartz
           function. The decay condition (\ref{an}) is essentially sharp}.

         This result may seem surprising, since when (\ref{an}) holds, the set $\La$ is "very close" to the  "limiting case" $\La=\Z$ when
          no generator exists.

          Since every $p_0$-spectrum $\La$ is also a $p$-spectrum  for every
          $p> p_0$ (see \cite{Bl}), the result above holds for  $p>2.$
          However, in fact  no perturbations are "needed" in this case:

         {\it The set of integers $\Z$ is a $p-$spectrum for all $p > 2 $}, see \cite{ao}.

         On the other hand, for $p=1$ the spectrum is never uniformly discrete.
          The following is proved in \cite{BOU}:

         {\it A set $\La$ is a $1$-spectrum if and only if the Beurling--Malliavin density
          of $\La$ is infinite}.

          Notice, so far $p=1$ is the only case when an effective characterization
           of $p$-spectra is known.

          A  detailed survey on the topic, including the results above, may be
          found in \cite{ou}.

\medskip\noindent{\bf 3.  The case ${\bf 1<p<2}$}.          This case is much less investigated.  In particular, it has been
          an open problem  whether   a $p$-spectral set can be
          uniformly discrete, or at least have a finite density.

      \section{Results}                    In this  note we prove the following

   \begin{theorem}
                                  There is a function $\varphi \in  S(\R)  $, such that for every set   $\La$ satisfying
                               (2) and (3), the set of translates
$ \{\varphi(t-\lambda),   \lambda\in \La\}$
                            spans the whole space $L^p(\R)$ for each $p>1$.\end{theorem}

                             The result holds for more general function spaces
                              "arbitrarily closed" to the space $L^1(\R)$, see the last section.

         The function  $\varphi$  will be  effectively constructed as the Fourier
          transform of a Schwartz function $\Phi$ with  "exponentially deep"  zeros
         at the integers and at infinity.

        The proof is based on a certain uniqueness theorem for tempered distributions.



\section{Functions with Deep Zeros}

In this section we present a uniqueness result for a class of analytic functions whose inverse Fourier transforms  have "deep zeros". Then in the next section this result will be extended to a certain class of  temperate distributions.

\begin{definition} Denote by $K_0$ the class of continuous functions $\Phi$ satisfying  the condition
\begin{equation}\label{dz}|\G(t)|\leq C_1 e^{-C_2(|t|+\frac{1}{d(t,\Z)})}, \quad  t\in\R,
\end{equation}where $d(t,\Z)=\min_{n\in\Z}|t-n|$, and $C_1=C_1(\Phi),C_2=C_2(\Phi)$ are positive constants depending on $\G$.
\end{definition}

Condition (\ref{dz}) means that $\G$ has "deep" zeros on the set of integers and at infinity.

Next, let $\hat K_0$ denote the class of the Fourier transforms $\varphi=\hat\Phi$ of functions  $\Phi\in K_0$.

Condition (\ref{dz}) implies that $\Phi$ has an exponential decay at $\pm\infty$. Hence, every element $\varphi\in\hat K_0$ is analytic and bounded in the horizontal strip $|$Im$\,z|<C,$ $0<C<C_2/(2\pi)$, where $C_2$ is the constant in (\ref{dz}). Clearly,  every derivative of $\varphi$ also belongs to $\hat K_0$.

\begin{md}    Suppose  $\La$ satisfies  (\ref{ais}) and (\ref{an}) and $\varphi\in \hat K_0$. If $\g|_\La=0$, then $\g=0$.
\end{md}

In other words, the sets $\La$ satisfying (\ref{ais}) and (\ref{an}) are uniqueness sets for the class $\hat K_0$.

\begin{proof}[Sketch of            Proof]         Assume  $\varphi\in\hat K_0$ and  $\g|_\La=0$ for some $\La$ satisfying (\ref{ais}) and (\ref{an}). We have to show that $\g=0$.

      Since $\varphi'$ is bounded on $\R$, from (\ref{ais}) and (\ref{an}) we get
\begin{equation}\label{ra}
                      |\g(n)| = O(r^{|n|}),\quad |n|\to\infty,
\end{equation}where $0<r< 1.$

                 Consider the periodization  $P(\G)$ of  $\G $ defined by

\[
                       P(\G) (t):= \sum_{k\in\Z} \G(t+k).
\]
                   The Fourier decomposition of $P(\G)$ has the form
\[
                       P(\G) (t)= \sum_{n\in\Z} \g(n) e^{2\pi int}.
\]
                  By (\ref{ra}),  $P(\G)$ is  analytic in some horizontal strip containing $\R$. On the other hand, it easily follows from  (\ref{dz}) that it has zero of infinite order at the origin.
                    Hence, $P(\G)= 0$,  so that we have
\[
                    \g(n)=0, \quad  n\in\Z.
\]
                Iterating this argument, one may prove that every derivative of $\g $ vanishes
               at the integers, and the result follows.                   See  details in [OU16], Lecture~11.

\end{proof}

    \section{Distributions with Deep Zeros}

              Let $S(\R)$ denote the space of Schwartz test-functions,
                and  $S'(\R)$ the space of tempered distributions on $\R$.
We will denote by $\langle F, \Phi\rangle$ the action of the distribution $F\in S'(\R)$ on the test function $\Phi\in S(\R)$.

We would like to extend condition (\ref{dz}) to tempered distributions.

Recall  that every tempered distribution $F\in\mathcal{S}'(\R)$ is the distributional derivative of finite order, $F=\Phi^{(k)}$,  of a continuous function   $\Phi$ having at most polynomial growth on the real axis (see \cite{FJ}, Theorem 8.3.1).

 \begin{definition} Denote by  $K$ the class of all tempered distributions  $F\in S'(\R)$ which  admit a representation \begin{equation}\label{dd}F=\Phi_1^{(k_1)}+\dots+\Phi_l^{(k_l)},\end{equation} where  $l\in\N,k_1,\dots,k_l\in\N\cup\{0\}$ and  $\Phi_1,\dots,\Phi_l$ are continuous functions satisfying (\ref{dz}).
\end{definition}

Let $\hat K$ denote the class the distributional Fourier transforms of the elements of $K$.

Assume $\Phi$ satisfies (\ref{dz}). Then the Fourier transform $\varphi=\hat\Phi$ is analytic and bounded in some horizontal strip. The distributional Fourier transform of $\Phi^{(k)}$ is the function
\[
\widehat{\Phi^{(k)}}= (2\pi ix )^k\varphi(x).
\]
It is therefore an analytic function of at most polynomial growth in the strip.

By the definition of class $\hat K$, we have
\begin{equation}\label{0}
x^k\varphi(x)\in \hat K,\quad \mbox{for every } k\geq0.
\end{equation}
From (\ref{dd}) it follows that   $\hat K$ consists of the functions $f$  admitting a representation
\[
f(x) = \sum_{j=1}^l(2\pi ix )^{k_j}\varphi_j(x), \quad \varphi_j=\hat\Phi_j, \Phi_j\in K_0,\quad j=1,\dots,l.
\]

 We will now list several simple properties of the class $\hat K$.

\begin{lemma} (a) $\hat K$ is a linear space over $\C$.

(b) Every element $f\in \hat K$ is analytic  and of at most polynomial growth in some strip $|$Im$\,z|<C, C=C(f)>0.$

(c) If $f\in \hat K$, then $xf(x)\in \hat K$ and $f'\in\hat K$.

(d) If $f\in \hat K$ then  Re$\, f(x)\in \hat K$ and Im$\, f(x)\in\hat K$.

(e) Assume $\varphi=\hat\Phi$, where  $\Phi\in S(\R)$ is such that  $\Phi,\Phi'\in K_0$.
Then for every $1\leq q\leq\infty$ and every $h\in L^q(\R)$ we have $\varphi\ast h\in \hat K.$

\end{lemma}

\begin{proof}
Statements (a) and (b) are obvious.

Clearly, it suffices to  prove statements  (c) and (d) for the case $l=1$ in (\ref{dd}), i.e. when $f(x)=\widehat{\Phi^{(k)}}=(2\pi i x)^k\varphi(x)$.

By  (\ref{0}), we see that $xf(x)\in K.$

Since the function $t\Phi(t)$ also satisfies (\ref{dz}), then
\[\varphi'= \widehat{(-2\pi i t)\Phi(t)}\in \hat K.\] From (\ref{0}) we get $x^k\varphi'(x)\in\hat K$ and $x^{k-1}\varphi(x)\in \hat K$ which yields
$f'\in\hat K$.  This proves (c).

Now,
write
\[\Phi(t):=\Phi_r(t)+i\Phi_i(t):=\frac{\Phi(t)+\Phi(-t)}{2}+i\frac{\Phi(t)-\Phi(-t)}{2i}.\]
Clearly, $\Phi_r$ and $\Phi_i$ satisfy (\ref{dz}). So, we have $\varphi=\varphi_r+i\varphi_i$, where the functions $\varphi_r=\hat \Phi_r\in\hat K$ and $\varphi_i=\hat\Phi_i\in\hat K$ are real on the real line. Hence, if $k$ is even, we see that
\[\mbox{Re}\,f(x)=(2\pi i x)^k\varphi_r(x)\in\hat K,\quad \mbox{Im}\,f(x)=(2\pi i x)^k\varphi_i(x)\in\hat K.\] Similarly,  one proves (d) for the odd $k$.

To prove (e), write $h(x)=h_1(x)+ h_2(x)$, where   \[ h_1(x):=\frac{h(x)}{1-2\pi i x}, \quad h_2(x):= -2\pi i x h_1(x).\] Hence,
$\varphi\ast h=\varphi\ast h_1+\varphi\ast  h_2.$

Clearly,   $h_1\in L^1(\R)$, and so  its  inverse Fourier transform  $H_1$ is continuous and bounded. This gives
 $\varphi\ast h_1= \widehat{\Phi H_1}\in\hat K$.

 Observe that $h_2=\widehat{H_1'}$, and so the inverse Fourier transform of
 $\varphi\ast h_2$ is the distribution $\Phi H_1'= (\Phi H_1)'- \Phi' H_1$. Since, by assumption, $\Phi'$ satisfies (\ref{dz}), we conclude that $\varphi\ast h_2\in\hat K$, which completes the proof of Lemma 1.
\end{proof}

\section{Uniqueness Theorem for the Class $\hat K$}

   \begin{theorem}   Assume $\La$  satisfies (\ref{ais}) and (\ref{an})
                and  $f \in\hat  K$.
                 If  $f|_\La =0$  then  $f=0$.
\end{theorem}

Hence, the sets $\La$ satisfying (\ref{ais}) and (\ref{an}) are uniqueness sets for the class $\hat K.$

\begin{ml} Assume $\La$ satisfies (\ref{ais}) and (\ref{an}). Then
\[
                    f\in\hat K, f |_\La= 0\Rightarrow f |_\Z=0 .
\]
\end{ml}

\begin{proof}  For convenience, in what follows we denote by $C$ different positive constants.

 Take any function $f\in\hat K$ which vanishes on $\La$.
 We have to show that $f|_\Z=0$.

 By Lemma 1, $f'\in K$ and  has at most polynomial growth on $\R$. So,   it follows from (\ref{ais}) and (\ref{an}) that \[ |f(n)|<C|n|^Cr^{|n|},\quad n\in\Z,\]
 where $0<r<1.$
 This shows that  the function
\[
R(t):=\sum_{n\in\Z}f(n)e^{2\pi i nt}, \quad t\in\R.
\]
is analytic in some strip $|$Im$\,z|<C$.

Clearly, to prove the lemma it suffices to  show that
$R(t)\equiv0$.
For this purpose we introduce the function $f_\epsilon:=fh_\epsilon$, where $\epsilon>0$ and
\[
h_\epsilon(x):=\left(\frac{\sin(2\pi \epsilon x)}{2\pi\epsilon x}\right)^N.
\]
 We will assume that $N$ is an even integer so large that  $f_\epsilon\in L^1(\R)$.

It is easy to see that $h_\epsilon=\hat H_\epsilon$, where
\[ H_\epsilon(t)= \left(\frac{1}{2\epsilon}1_{\epsilon}(t)\right)^{N*}.\] Here $1_\epsilon$ is the indicator function of $[-\epsilon,\epsilon]$. Hence, $H_\epsilon(t)\geq 0, t\in\R,$  and so
\[
\|H_\epsilon\|_1=h_\epsilon(0)=1.
\]

Given  $d>0$,  we denote by $S(d)$ the subspace of $S(\R)$ of functions $\Psi$ vanishing outside $(-d,d).$
In what follows we  assume that $0<d<1/2$ and $\epsilon>0$ is so small  that $d+N\epsilon<1/2$.

Fix any function  $\Psi\in S(-d,d)$. Then  $$\Psi\ast H_\epsilon\in S(d+N\epsilon),$$ and we have
\begin{equation}\label{hh}
\|(\Psi\ast H_\epsilon)^{(k)}\|_\infty\leq \|\Psi^{(k)}\|_\infty\|H_\epsilon\|_1=\|\Psi^{(k)}\|_\infty.
\end{equation}

Denote by $F_\epsilon$ the inverse Fourier transform of $f_\epsilon$, and by $\psi$ the Fourier transform of $\Psi$. From (\ref{dd}), for every $n\in\Z$ we obtain:
\[
|\langle F_\epsilon,\Psi(t-n)\rangle|=|\langle f_\epsilon,e^{2\pi inx}\psi\rangle|=|\langle f,e^{2\pi inx}\psi h_\epsilon\rangle|=\]\[
|\langle \sum_{j=1}^l\Phi_j^{(k_j)},(\Psi\ast H_\epsilon)(t-n)\rangle|
\leq\|(\Psi\ast H_\epsilon)^{(k)}\|_\infty\sum_{j=1}^l\int_{n-d-N\epsilon}^{n+d+N\epsilon}|\Phi_j|,\] where $k:=\max_j{k_j}$.
Recalling that the functions $\Phi_j$  satisfy (\ref{dz}), by  (\ref{hh}) we obtain
\begin{equation}\label{4}
|\langle F_\epsilon,\Psi(t-n)\rangle|\leq Ce^{-C(|n|+\frac{1}{N\epsilon+d})}\| \Psi^{(k)}\|_{\infty}, \quad n\in\Z.
\end{equation}

Set
\[
R_\epsilon(t):=\sum_{n\in\Z}f_\epsilon(n)e^{2\pi i n t},
\]
and let us calculate the product $\langle R_\epsilon,\Psi\rangle.$

By the Poisson formula,
\[
R_\epsilon(t)=\sum_{n\in\Z}f_\epsilon(n)e^{2\pi i n t}= \sum_{n\in\Z}F_\epsilon(t+n).
\]
Therefore, by (\ref{4}),
\[
|\langle R_\epsilon, \Psi\rangle|=|\langle\sum_{n\in\Z}F_\epsilon(t+n), \Psi(t)\rangle|=|\langle F_\epsilon(t),\sum_{n\in\Z} \Psi(t-n)\rangle|
\]
\[
\leq C\| \Psi^{(k)}\|_{\infty}e^{-\frac{C}{N\epsilon+d}}.
\]
Letting $\epsilon\to 0$  we arrive at the inequality:
\begin{equation}\label{5}
|\langle R, \psi\rangle| \leq C\|\Psi^{(k)}\|_{\infty}e^{-\frac{C}{d}},
\end{equation}  which holds for some  $k$ and every $0< d<1/2$ and  $ \Psi\in S(d)$.

Now, the Main Lemma follows from

\begin{lemma} Assume a function $R$ is  analytic  in a strip $|${\rm Im}$\,z|<C$  and  satisfies (\ref{5}). Then $R=0$.\end{lemma}

The proof is simple: One may use the standard calculus to check that $R$ must have zero of infinite order at the origin, and so $R=0$. We omit the details.
\end{proof}

\section{Proofs of Theorem 1 and 2}

\begin{proof}[Proof of Theorem 2] Assume $\La$ satisfies (1) and (2), $f\in\hat K$  and $f|_\La=0$. Write $f=f_r+if_i,$ where $f_r(x):=$Re$\,f(x)$ and $f_i(x):=$Im$\, f(x)$.

 Let us show that $f_r=0$. This function vanishes on $\La$ and by Lemma 1,  $f_r\in K$.  It follows from the Main Lemma that $f_r|_\Z=0$.

 Since $f_r$  is  real  for real $x$ and vanishes both on $\Z$  and on $\La=\{n+a_n, n\in\Z\}, a_n\ne0$,  its derivative $f'_r$  vanishes on some set $\La_1:=\{n+a_1^{(1)}\}$, where each point $a^{(1)}_n$ lies inside the open interval between $0$ and $a_n$. Since $\La$ satisfies (\ref{ais}) and (\ref{an}),  so does   $\La_1$.
  By Lemma 1, $f_r'\in  K.$ We now we apply the Main Lemma again to get $f_r'|_\Z=0$. A straightforward iteration of this argument proves that $f_r^{(j)}|_\Z=0$ for every $j\in\N$.  Since $f_r$ is analytic in some strip $|$Im$\,z|<C$, we conclude that $f_r=0.$ Similarly, one proves that $f_i=0$, and so $f=0.$
\end{proof}

\begin{proof}[Proof of Theorem 1]
Choose any real function $\Phi\in S(\R)$ such that $\Phi$ and $\Phi '$ satisfy (\ref{dz}) and such that $\Phi(-t)=\Phi(t)>0$ for $t\not\in\Z$. Then its Fourier   transform $\varphi$ is real and $\varphi(-x)=\varphi(x), x\in\R$.

Let $\La$ satisfy (\ref{ais}) and (\ref{an}).
 Assume the set of  translates  $\{\varphi(x-\lambda),\lambda\in\La\}$ does not span $L^p(\R)$. Then there is a non-trivial function $h\in L^{q}(\R), 1/p+1/q=1,$ such that $$\int_\R \varphi(x-\lambda)h(x)\,dx= (\varphi\ast h)(\lambda)=0,\quad\lambda\in\La.$$ By Lemma 1, $\varphi\ast h\in\hat  K$.  Theorem 2 yields  $\varphi \ast h=0.$ This means that all translates  $\{\varphi(x-s), s\in\R\}$ do not span the space $L^p(\R)$. However, this  contradicts Beurling's theorem cited in the introduction.
 Theorem 1 is proved.
\end{proof}

\section{Extension}
Below we assume that $X$ is a Banach function space on $\R$ with the property that  the Schwartz  space $S(\R)$ is continuously embedded  and dense in $X$.

The following extension of Theorem 1 hods true:

\begin{theorem}    Assume that no non-trivial distribution  $h\in X^*$ has spectrum on $\Z$.
Then there exists $\varphi\in S(\R)$ such that for every set   $\La$ satisfying
                               (2) and (3), the set of translates
$$ \{\varphi(t-\lambda),   \lambda\in \La\}$$
                            spans $X$.
\end{theorem}

\begin{proof}  Choose a Schwartz function $\varphi$ as in the proof of Theorem 1, where we additionally require that  {\it every derivative}  $\Phi^{(j)}$  satisfies (3).  If the $\La$-translates of $\varphi$ do not span $X$, then  there is a distribution $h\in X^\ast$   such that the convolution $\varphi * h$ vanishes on $\La$. The distributional inverse Fourier transform of  $\varphi * h$ is the the distribution $\Phi H,$ where $h=\hat H.$ Since  $H$ is the distributional derivative of finite order of some continuous function, similarly to Lemma 1 (e) one may check that $\varphi * h\in \hat K$. Hence, by Theorem 2,
$$
                           \Phi H  = 0.
$$
It follows that the support of  $H$ lies in $\Z$   (the zero set of $\Phi$). Contradiction.\end{proof}

  \begin{exa}   Let $w>0$ be a weight  vanishing at infinity.
                      Then Theorem 3 holds for the space
$$
                          X=   L^1(w, \R) := \{f:    \|f\|_w= \int_\R | f(x)|w(x)\,dx < \infty\}.
$$\end{exa}
                 Indeed, the elements of the dual space  $X^\ast=L^\infty(1/w, \R)$ are functions vanishing at infinity. One may prove that no non-trivial function $h\in X^\ast$   may have spectrum on $\Z$.

   However, as mentioned in Introduction,  Theorem 3 ceases to be true for $w=1 $, that is for the space $L^1(\R)$.

       Theorem 3 is also applicable to the separable Lorentz spaces and also to more general symmetric spaces, see definition of these spaces in \cite{kps}, Ch. 2.


\end{large}
\end{document}